\documentclass[12pt]{amsart}

\usepackage[iso-8859-7]{inputenc}
\usepackage{amsthm}
\theoremstyle{plain}
\newtheorem*{theorem*}{Theorem}
\usepackage{amsmath}
\usepackage{xcolor}
\usepackage{verbatim}

\newtheorem{theorem}{Theorem}
\newtheorem{lemma}{Lemma}
\newtheorem{corollary}{Corollary}[section]
\theoremstyle{definition}

\newtheorem{proposition}{Proposition}
\theoremstyle{remark}
\newtheorem{remark}{Remark}

\numberwithin{equation}{section}

\newcommand{\abs}[1]{\lvert#1\rvert}
               \def\g{\gamma}
\def\d{\delta}            
              
                  \def\z{\zeta}
               \def\vf{\varphi}

\def\e{\varepsilon}       \def\vf{\varphi}

\def\D{{\mathbb D}}  
\def\C{{\mathbb C}}

\def\({\left(}       \def\){\right)}

\newcommand{\n}[1]{\Vert#1\Vert}

\newtheorem{thmx}{\bf Theorem}

\usepackage{enumerate}
\numberwithin{equation}{section}
\usepackage{mathtools}
\usepackage{amsmath,amssymb}
\usepackage[a4paper, left=3.5cm,right=3.5cm,top=2.5 cm,bottom=3 cm,]{geometry}
\DeclarePairedDelimiter{\norm}{\lVert}{\rVert}

\date{\today}

\begin{document}

\title[Composition Semigroups and Integral Operators in $BMOA_p$ spaces]{Semigroups of Composition
operators and Integral operators in BMOA-type spaces}
\author{V. Daskalogiannis}
\address{Dept. of Mathematics, Aristotle University of Thessaloniki, 54124, Greece }
\email{V. Daskalogiannis: vdaskalo@math.auth.gr}
\author{P. Galanopoulos}
\address{Dept. of Mathematics, Aristotle University of Thessaloniki, 54124, Greece }
\email{P. Galanopoulos: petrosgala@math.auth.gr}


\begin{abstract}
The aim of this article is to study semigroups of composition operators $T_t=f\circ\phi_t$ on the BMOA-type spaces
$BMOA_p$, and on their "little oh" analogues  $VMOA_p$. The spaces $BMOA_p$ were  introduced by R. Zhao
as part of the large family of $F(p,q,s)$  spaces, and are the M\"{o}bius invariant subspaces of the
Dirichlet spaces $D^p_{p-1}$.

We study the  maximal subspace $[\phi_t, BMOA_p]$ of strong
 continuity, providing a sufficient condition on the infinitesimal generator of $\{\phi_t\}$, under which $[\phi_t,
 BMOA_p]=VMOA_p$, and a related necessary condition in the case where the Denjoy - Wolff point of the
  semigroup is in $\D$.  Further,  we  characterize those semigroups, for which $[\phi_t, BMOA_p]=VMOA_p$,
  in terms of the resolvent operator of the  infinitesimal generator  of $(T_t|_{VMOA_p})$.  In addition we provide a  connection between the maximal  subspace of strong  continuity and the  Volterra-type operators $T_g$. We characterize the symbols $g$ for which $T_g:\,BMOA\to BMOA_1$ is bounded or  compact, thus extending a related result to the case $p=1$. We also prove that for $1<p<2$ compactness of $T_g$ on $BMOA_p$ is  equivalent to weak compactness.
\end{abstract}

\maketitle

\section{Introduction}
Let $\D$ be the unit disc of the complex plane and $\partial \D$ its boundary.
We will denote by $H(\D)$ the space of all analytic functions on the disc and by $H^\infty$
the subspace of bounded analytic functions. A family $\{\phi_t\}_{t\geq 0}$ of analytic
self-maps of the disc is a semigroup of functions if:

\begin{enumerate}
\item $\phi_0(z)=z$

\item $\phi_t\circ \phi_s=\phi_{t+s}$,\; for $t,\, s\,\geq\,0$

\item $\phi_t \to \phi_0$, uniformly on compact subsets of $\D$, as $t\to 0$.
\end{enumerate}
The identity maps $\,\phi_t(z)=z,\;t\geq 0 ,\,$ form the trivial semigroup. In any other case we
say $\{ \phi_t \}_{t\geq 0}$ is
nontrivial. Each family $\{\phi_t\}_{t\geq 0}\,$ induces a semigroup of composition operators on $H(\D)$,
\[
T_t(f)\,=\,f\circ \phi_t\,,\; t\geq0\,.
\]
If   $X$ is   a Banach space of analytic functions on $\D$ on which the composition operators   $T_t: X\to X$
are bounded, there arises the question whether $(T_t)$ is strongly continuous, that is if
\[
\lim_{t \to 0^+}\, \norm{T_t(f) - f}_X\,=\,0,\;\;f\in X.
\]

This  question of strong continuity  was first studied by E. Berkson and H. Porta \cite{Berkson}
on Hardy spaces. They proved that each $\{\phi_t\}$ induces a strongly
continuous composition semigroup $(T_t)$  on each Hardy space ${H^p,\; 1 \leq p < \infty}$.  Several other
 authors studied the question  on classical spaces of
analytic functions, in particular on Bergman spaces, the Dirichlet space, the space BMOA and the Bloch
space among others.

This article studies strong continuity of composition semigroups in the family of spaces $BMOA_p,\;1\leq p<\infty$.
 When $p\geq 2$ it is immediate that   every semigroup
$\{\phi_t\}$ induces a semigroup $(T_t)$ of bounded composition operators on $BMOA_p$. We show that the same is true in
the non-obvious case $1\leq p<2$ and also for $VMOA_p$. We then study the maximal subspace $[\phi_t, BMOA_p]$ of
strong continuity on these spaces. We show that $[\phi_t, BMOA_p]\subsetneqq BMOA_p$ for nontrivial $\{\phi_t\}$ and
 provide  conditions under which  $[\phi_t, BMOA_p]=VMOA_p$. We also  consider Volterra-type
  operators $T_g$ on the above spaces  and  characterize the symbols $g$ for which $T_g:\,BMOA\to BMOA_1$ is
bounded or  compact, extending a  result from   \cite{YuTo} to the case $p=1$. The key point for proving this extension is the use of the Garsia norm for functions in $BMOA$. 
 Finally we prove that for  $1<p< 2$
compactness of $T_g$ on $BMOA_p$ is equivalent to its weak  compactness, and find a connection between
 the  maximal subspace of strong continuity and  the mapping properties
 of $T_{\gamma}$ for a specific symbol $\gamma$ determined by the inducing semigroup.

\section{Background and Main Results}

 \subsection{The spaces $BMOA_p$} Let $0<p<\infty$. A function $f\in H(\D)$
 belongs to the space $BMOA_p$ if
\begin{equation}\label{def}
\sup_{a\in\D}\, \int_{\D} | f'(z)|^p\,(1-|z|^2)^{p-2}(1-|\vf_a(z)|^2)\,dm(z) < \infty,
\end{equation}
where $\vf_a(z)= \frac{a-z}{1-\bar{a}z}$, $a\in\D$,
are the M\"{o}bius automorphisms of the disc and  $dm(z)=rdrd\theta/\pi$ is
the normalized  Lebesgue area
measure of $\D$.
The subspace $VMOA_p$ contains those $f$ for which
\begin{equation}\label{vmoa}
\lim_{| a|\to 1}\int_{\D} | f'(z)|^p(1-| z |^2)^{p-2}(1-|\vf_a(z)|^2) \,dm(z) = 0\,.
\end{equation}

The spaces $BMOA_p$ are part of the large family of spaces $F(p, q, s)$   introduced by R. Zhao in \cite{Zhao}.
An  $f\in H(\D)$ is in $F(p, q, s)$ if
\[
\sup_{a\in\D}\, \int_{\D} | f'(z)|^p(1-| z|^2)^q  (1-|\vf_a(z)|^2)^s\,dm(z) < \infty\,,
\]
and $f\in  F_0(p, q, s)$  if the above integral   tends to $0$ as $|a|\to 1$. The parameters are $p>0,\, q>-2,\, s>0$,
and the spaces are nontrivial when $q+s>-1$.
These spaces  were  further studied by J. R\"{a}tty\"{a} in \cite{Ratt}.

Clearly
$$
BMOA_p \equiv F(p,\, p-2,\,1) \;\;\textrm{and}\;\;VMOA_p \equiv F_0(p,\, p-2,\,1),
$$
and   the value $p=2$  corresponds to the  classical spaces $BMOA$  and $VMOA$.

We recount  some basic properties of $BMOA_p$  most of which can be found in  \cite{Zhao}. For $p\geq 1$, $BMOA_p$ equipped with the norm defined by
\begin{equation}\label{norm1}
\n{f}_{BMOA_p}^p=|f(0)|^p + \sup_{a\in\D}\int_{\D} |f'(z)|^p(1-|z|^2)^{p-2}(1-|\vf_a(z)|^2) \,dm(z)
\end{equation}
is a Banach space, with  $VMOA_p$ being the closure  of the
polynomials in $BMOA_p$.   The identity
$$
1-| \vf_a(z) |^2= |\vf_a'(z) | (1-| z|^2), \quad a,z\in \D,
$$
and a change of variables $z=\vf_a(w)$  gives
\begin{equation}\label{norm2}
\n{f}_{BMOA_p}^p=|f(0)|^p+ \sup_{a\in\D}\n{f\circ \vf_a - f(a)}^p_{D^p_{p-1}}.
\end{equation}
where
$$
\n{f}^p_{D^p_{p-1}}=|f(0)|^p + \int_{\D} |f'(z)|^p (1-|z|^2)^{p-1}dm(z)
$$ is a norm defining  the Dirichlet
spaces $D^p_{p-1}:=\{f\in H(\D): \n{f}_{D^p_{p-1}}<\infty\}$. Thus (\ref{norm2})  says that $f\in BMOA_p$ if and only if
the set of hyperbolic translates
$$
\left\{f\circ \vf_a-f(a): a\in \D\right\}
$$
is bounded in  $D^p_{p-1}$, i.e. $BMOA_p$ is the M\"{o}bius invariant version of $D^p_{p-1}$.  For more information on
$D^p_{p-1}$ see \cite{GiPe}, \cite{Vin}.

It is to be noticed at this point that the spaces $D^p_{p-1}$ and $BMOA_p$ are part of the  larger families  of spaces
$D^p_{\mu}$ and their M\"{o}bius invariant subspaces   $M(D^p_{\mu})$ respectively, which were studied  by M. D. Contreras, S.
D\'iaz-Madrigal and D. Vukoti\'c in \cite{cdmv}. Specifically, given a positive Borel measure $\mu$ on $\D$ consider the
Dirichlet-type spaces
$$
D^p_{\mu}=\{f\in H(\D): \int_{\D}|f'(z)|^p\,d\mu(z)<\infty\}
$$
and their Mobius invariant versions
$$
M(D^p_{\mu})=\{f\in D^p_{\mu}:\sup_{a\in \D}\int_{\D}|(f\circ\vf_a)'(z)|^p\,d\mu(z)<\infty\},
$$
along with the little-oh spaces $ M_0(D^p_{\mu})$ containing those $f$ for which
$$
\lim_{|a|\to 1}\int_{\D}|(f\circ\vf_a)'(z)|^p\,d\mu(z)=0.
$$
The choice   $d\mu(z)=(1-|z|^2)^{p-1}dm(z)$ gives $D^p_{p-1}=D^p_{\mu}$,  and then 
\[
BMOA_p=M(D^p_{\mu}) \textrm{ and } VMOA_p=M_0(D^p_{\mu}).
\]

Even though there is  no inclusion relation between  $D^p_{p-1}$ and $ D^q_{q-1}$ for $p\neq q$, their M\"{o}bius invariant subspaces   $BMOA_p$ form a chain that increases in size with $p$, \cite{Zhao}. If $0<p<q<\infty$ then
$$
BMOA_p\subsetneq  BMOA_q \subsetneq \mathcal{B},
$$
where $\mathcal{B}$ is the Bloch space of $f\in H(\D)$  with $\sup\limits_{z\in\D}(1-|z|^2)|f'(z)|<\infty$. An analogous
containment relation  is valid for $VMOA_p$. For the sake of completeness we prove the following lemma which can also be proved using the closed graph theorem.
\begin{lemma}
\begin{enumerate}
\item For   each  $0<q<\infty$ there is a constant $C=C_q$ such that
 $\n{f}_{\mathcal{B}}\leq C\n{f}_{BMOA_q}$, for all $f\in BMOA_q$.\\
\item For  each pair $(p,q)$ with $0<p<q<\infty$ there is a constant $C=C_{p,q}$ such that $\n{f}_{BMOA_q} \leq   C\n{f}_{BMOA_p}$, 
 for all $f\in BMOA_p$.
 \end{enumerate}
\end{lemma}

\begin{proof}
(1) For  $f\in BMOA_q$ and $a\in \D$ let
$$
I(f, q, a)= \int_{\D}|f'(z)|^q(1-|z|^2)^{q-2}(1-|\vf_a(z)|^2)\,dm(z).
$$
According to \cite[Lemma 2.9]{Zhao} there is a constant $K_q$ such that
$$
|f'(a)|^q(1-|a|^2)^q \leq K_q I(f, q, a),
$$
for all $a\in \D$. Thus we have
\begin{align*}
\left (|f(0)|+|f'(a)|(1-|a|^2)\right)^q &\leq 2^q\left(|f(0)|^q+|f'(a)|^q(1-|a|^2)^q\right)\\
&\leq 2^q\left(|f(0)|^q+K_qI(f, q, a)\right)\\
&\leq 2^q\max\{1, K_q\}(|f(0)|^q+I(f, q, a))\\
&\leq K_q'\n{f}_{BMOA_q}^q,
\end{align*}
and taking the $\sup$ on $a$ of the quantity in left hand side, we have the desired inequality with $C_q=(K_q')^{1/q}$.

(2) Let $f\in BMOA_p$   then for $I(f, q, a)$ as defined above we have
\begin{align*}
I(f, q, a)&=\int_{\D} \left((1-|z|^2)|f'(z)|\right)^p\left((1-|z|^2)|f'(z)|\right)^{q-p}
\frac{(1-|\vf_a(z)|^2)}{(1-|z|^2)^2}  \,dm(z)\\
&\leq  \sup_{z\in \D}\left((1-|z|^2)|f'(z)|\right)^{q-p}  \int_{\D} |f'(z)|^p(1-|z|^2)^p
\frac{(1-|\vf_a(z)|^2)}{(1-|z|^2)^2}  \,dm(z)\\
&\leq \n{f}_{\mathcal{B}}^{q-p}\int_{\D} |f'(z)|^p(1-|z|^2)^{p-2}(1-|\vf_a(z)|^2)  \,dm(z)\\
&\leq \n{f}_{\mathcal{B}}^{q-p} \n{f}_{BMOA_p}^p\\
&\leq C^{q-p}\n{f}_{BMOA_p}^{q-p}\n{f}_{BMOA_p}^p, \qquad (\text{by using (i)})\\
&=C^{q-p}\n{f}_{BMOA_p}^q.
\end{align*}
Thus
\begin{align*}
\n{f}_{BMOA_q}^p&=\left(|f(0)|^q+\sup_{a\in \D}I(f, q, a)\right)^{p/q}\\
&\leq 2^{p/q}\left(|f(0)|^p+\sup_{a\in \D}I(f, q, a)^{p/q}\right)\\
&\leq 2^{p/q}\left(\n{f}_{BMOA_p}^p+\left( C^{q-p}\n{f}_{BMOA_p}^q\right)^{p/q}\right)\\
&= 2^{p/q}(1+C^{\frac{(q-p)p}{q}})\n{f}_{BMOA_p}^p,
\end{align*}
and the conclusion follows with $C_{p,q}=(2^{p/q}(1+C^{\frac{(q-p)p}{q}}))^{1/q}$.
\end{proof}

Since $H^{\infty}\subset BMOA=BMOA_2$ it follows that
 $H^{\infty}\subset BMOA_p$ for $p\geq 2$,
but this containment is no longer true for $1\leq p<2$. In fact it is shown in  \cite[Proposition A
(7)]{PerRat} and  \cite[Theorem 2.3.4]{Ratt} that the disc algebra $\mathcal{A}$ is not contained in $ BMOA_p$ for
$p<2$.

Yet another  description of the spaces $BMOA_p$ is in terms of Carleson measures. Recall that a positive
measure $\mu$ on $\D$ is a Carleson measure if
\begin{equation}\label{defCar}
\n{\mu}_{CM}=\sup_{I\subset \partial\D} \frac{\mu(S(I))}{|I|} <\infty,
\end{equation}
where $I\subset \partial\D$ is an arc with length $|I|$,    and
$$
S(I)= \{re^{it}: e^{it}\in I, \, \text{and} \, 1-|I|<r<1\}
$$
is the Carleson square for  $I$. Further, $\mu$ is a vanishing Carleson measure if $\lim\limits_{|I|\to 0}\frac{\mu(S(I))}{|I|}=0$.
The defining property (\ref{defCar}) is equivalent to that
the Hardy space $H^2$ is boundedly embedded  in $L^2(\D, \mu)$, i.e. there is a constant $C_{\mu}$ such that
\begin{equation}\label{CarH2}
\int_{\D}|f(z)|^2\,d\mu(z)\leq C_{\mu}\n{f}_{H^2}^2
\end{equation}
for all $f\in H^2$. Vanishing Carleson measures are those for which the embedding  $i:H^2\to L^2(\D, \mu)$
is a compact operator.
It is shown in \cite{Zhao} that $f\in BMOA_p$ if and only if
$$
\sup_{I\subset \partial \D} \frac{1}{|I|}\int_{S(I)} |f'(z)|^p (1-|z|^2)^{p-1}\,dm(z)<\infty,
$$
equivalently the measure $ d\mu(z)=|f'(z)|^p(1-|z|^2)^{p-1}\,dm(z)$  is a Carleson measure on $\D$.
$VMOA_p$ then consists of those $f$ for which $d\mu$ is vanishing Carleson. A consequence is that
 $$
 |f(0)|^p+\sup_{I\subset \partial\D} \frac{\mu(S(I))}{|I|}\simeq\n{f}_{BMOA_p}^p.
 $$

 Using the above description, Z. Wu \cite{Wu} (see also \cite{Vin}) gave a characterization
 of pointwise multipliers of $D^p_{p-1}$, i.e. of functions in the space
$$
\mathcal{M}(D^p_{p-1})=\{g\in H(\D): gf\in D^p_{p-1} \,\,\text{for each} \, f\in D^p_{p-1}\}.
$$
This characterization  for $0<p\leq 2$ can be restated as follows
$$
g\in \mathcal{M}(D^p_{p-1}) \Leftrightarrow g\in H^{\infty}\cap BMOA_p.
$$

Another important class of spaces contained in family $F(p,q,s)$ are the $Q_s$ spaces, obtained as
$Q_s = F(2, 0, s)$ and $Q_{s,0}=F_0(2, 0, s)$.  They were
introduced by Aulaskari, Xiao and Zhao in \cite{AuXiZh}, see also  \cite{AuLa}. For $s=1$ we have  $Q_1= BMOA$, while
for all $s>1$, $Q_s=\mathcal{B}$. For $0<s<1$ these spaces are distinct and they increase in size   with $s$.
In \cite{AuTo}, Aulaskari and Tovar proved that
\[
 \bigcup_{s\in(0,1)} Q_s \subsetneq \bigcap_{0<p\leq2} BMOA_p,
\]
which indicates that $BMOA_p$ are considerably larger spaces than $Q_s$. A good reference on $Q_s$ spaces is Xiao's
monographs \cite{Xiao1} and \cite{Xiao2}.

 We note that in the literature, the spaces $BMOA_p$  are also referred to as
Besov-type spaces, see for example \cite{AuTo} or \cite{PerRat2}.

\subsection{Semigroups of composition operators}

 If $\{\phi_t\}_{t\geq 0}$ is a semigroup then each $\phi_t$ is univalent and the limit
\[
G(z)=\lim_{t\to 0^+}\frac{\phi_t(z)-z}{t}
\]
exists uniformly on compact subsets of $\D$. The analytic function $G(z)$ is
the \emph{infinitesimal generator} of $\,\{\phi_t\}\,$ and uniquely determines
the semigroup. Moreover $G$ satisfies
\begin{equation}\label{gen}
G(\phi_t(z))=\frac{\partial \phi_t(z)}{\partial t}\,=\,G(z)\,\frac{\partial \phi_t(z)}{\partial z},\;\;z\in\D\,,\;\;t\geq 0\,.
\end{equation}
The infinitesimal generator $G$ can be uniquely represented in terms of the Denjoy-Wolff point $b$ of
the semigroup (see \cite{Berkson}) as
\begin{equation}\label{D-W}
G(z)=(\bar{b}z -1)(z-b)P(z),\;\;z\in\D,
\end{equation}
where $b\in \overline{\D}$ and $P\in H(\D)$ with $Re\left(P(z)\right)\geq 0$ for all $z\in\D$ .

If $X$ is a Banach space consisting of analytic functions on $\D$ we denote by  $[\phi_t, X]$  the maximal closed subspace of
$X$ on which $\{\phi_t\}$ generates a strongly continuous composition semigroup $(T_t)$, that is,
$$
[\phi_t,\,X]\coloneqq \{ f\in X:\; \lim_{t\to 0^+}\n{f\circ \phi_t - f}_X =0 \}.
$$
It is shown in \cite{bcdmps} that if $X$ contains the constants and $\sup\limits_{0<t<1}\n{T_t}<\infty$, then
\begin{equation}\label{ssc}
[\phi_t,\,X]= \overline{\{f\in X: Gf'\in X\}}
\end{equation}
with $G$ the generator of $\{\phi_t\}$.
It is well known that for every semigroup $\{\phi_t\}$,
$$
[\phi_t,\,X]=X
$$
when $X$ is any of the following spaces: the Hardy spaces $H^p,\;1\leq p<\infty$
\cite{Berkson}, the Bergman spaces $A^p_a,\;1\leq p<\infty,\;a>-1$ \cite{Sisk3}, the Dirichlet space $\mathcal{D}$
\cite{Sisk4}, the space VMOA and  the little Bloch space $\mathcal{B}_0$ \cite{bcdmps},
and $Q_{s,0}$ \cite{WirthXiao}.
 Note that in each of the above spaces, polynomials are dense in $X$, a property
  that plays a key role in proving strong continuity.

This is no longer true when $X=H^\infty$ or $X=\mathcal{B}$. For these spaces and for nontrivial  $\{\phi_t\}$
we have,
\[
[\phi_t, H^\infty]\subsetneqq H^\infty \;\;\textrm{and}\;\;[\phi_t,\mathcal{B} ]\subsetneqq \mathcal{B},
\]
see for example \cite{bcdmps},\cite{Sisk}.  Extending this result, A. Anderson, M. Jovovic and W. Smith
 proved in \cite{AJS} that
whenever $X$ is a space such that
\[
H^\infty\subseteq X \subseteq \mathcal{B}
\]
then $[\phi_t , X] \subsetneqq X$. In  particular for $X=BMOA_p$, $p\geq 2$,
$$
[\phi_t , BMOA_p] \subsetneqq BMOA_p
$$
for all nontrivial $\{\phi_t\}$, since for such $p$ we have  $ H^{\infty}\subset BMOA_p\subset \mathcal{B}.$
However if  $1\leq p<2$ then $ H^\infty \not\subset BMOA_p$, so the
result of Anderson, Jovovic and Smith does not apply for this range of $p$ and the question of strong continuity needs a
different treatment.

On the other hand using the fact that the closure of the polynomials in $BMOA_p$ is $VMOA_p$ we will show
that each $\{\phi_t\}$ induces a strongly
continuous composition semigroup on $VMOA_p$, so  we  have
$$
VMOA_p\subseteq [\phi_t , BMOA_p]
$$
for $p\geq 1$. This containment can be proper as the following example shows.
Let
\[
f(z)=\log(1-z) \in BMOA_p\setminus VMOA_p
\]
 and
\[
 \phi_t(z) = e^{-t}z +1 -e^{-t},\quad t\geq 0.
 \]
Then  $ \lim_{t\to 0}\,\norm{f\circ \phi_t - f}_{BMOA_p}=0$, thus $f\in  [\phi_t , BMOA_p]$  so, in this case,
$$
VMOA_p\subsetneqq [\phi_t , BMOA_p].
$$
 Thus for a general semigroup $\{\phi_t\}$ the question arises to relate function-theoretic properties
of $\{\phi_t\}$ to the size of the maximal subspace of strong continuity of $(T_t)$.

\subsection{Main results}
In section 3 we study the boundedness of composition operators
$$
C_\psi (f) (z)=f (\psi (z)),
$$
 on $BMOA_p$ and $VMOA_p$, where $\psi$ is an analytic self-map of the disc.  It is well known that $C_\psi$
is  bounded on BMOA \cite{bmoa} and on $\mathcal{B}$ for all $\psi$, and it is not difficult to  check
that this remains true for  $BMOA_p$ for $p>2$. When $1\leq p<2$ we prove that if   $\psi$ is univalent
then  $C_{\psi}$  is bounded on $BMOA_p$ and on $VMOA_p$, and this  is sufficient for
our work with semigroups.

In section 4, we study semigroups of composition operators on $BMOA_p$ and $VMOA_p$, $p\in[1, 2)$.
It is easy to see that  for  the dilations $\phi_t(z) =e^{-t}z$ or the
 rotations $\phi_t=e^{it}z$, we have a proper inclusion
\[
[\phi_t, BMOA_p] \subsetneqq  BMOA_p.
\]
Indeed taking  any function $g\in BMOA_p\setminus VMOA$,
if we assume that $g \in [\phi_t, BMOA_p]$ then  $g \in [\phi_t, BMOA]$, since
$$
\lim_{t\to 0}\,\n{g\circ\phi_t - g}_{BMOA}\leq C \lim_{t\to 0}\, \n{g\circ\phi_t - g}_{BMOA_p},\quad 1\leq p <2.
$$
Then, by Sarason's theorem \cite{Sar}, $g \in VMOA$ which is a contradiction.
Examples of such functions are $g(z)=\log(1-z)$, and the class of inner functions
\[
g(z)=\exp\left( \gamma\,\frac{z+w}{z-w} \right)\,,
\]
where $0<\gamma<\infty$ and $w\in \partial \D$ (see \cite[Theorem 1.1]{PerRat}).
More generally using an argument similar to that in \cite{AJS} we will show that  for any nontrivial $\{\phi_t\}$
we always  have
$$
[\phi_t, BMOA_p] \subsetneqq  BMOA_p, \quad 1\leq p<2.
$$

On the other hand we will show that if  the infinitesimal generator $G$ of $\{\phi_t\}$ satisfies the condition
\begin{equation*}
\lim_{\vert I \vert \to 0} \frac{\left(\log\frac{2}{\vert I \vert}\right)^p}{\vert I \vert}
\int_{S(I)} \frac{(1-\vert z\vert^2)^{p-1}}{\vert G(z) \vert^p}\,dm(z)\,=\,0\,,
\end{equation*}
then $VMOA_p = [\phi_t, BMOA_p]$, and will observe that this condition is satisfied for a large family of semigroups.
We also prove  a (different) necessary condition for this equality, for semigroups with
Denjoy-Wolff point inside $\D$.

We finish the section by observing that  a characterization of the equality $[\phi_t, BMOA_p]=VMOA_p$ may be given in terms of
the  resolvent operator $\mathcal{R} (\lambda, \Gamma)=(\lambda-\Gamma)^{-1}$ of the infinitesimal generator $\Gamma$ of
the semigroup $(T_t)$ acting on $VMOA_p$ when  $1<p\leq 2$, and that this is  equivalent to
 the weak compactness of  $\mathcal{R}(\lambda, \Gamma)$  on $VMOA_p$.

In section 5, we study the compactness and weak compactness of the Volterra-type operators
\[ T_g=\int_0^z f(\zeta) g'(\zeta)\,d\zeta \]
on $BMOA_p$ and $VMOA_p$. We give a condition on $g$ that characterizes when $T_g:BMOA \to BMOA_1 $ is bounded or compact
thus extending \cite[Theorem 14 and  Corollary 16 (3)]{YuTo} to the case $p=1$. To prove this extension, we consider the Garsia norm (\cite{Girela}) for functions in $BMOA$.
 In addition we give equivalent conditions for $T_g$ to be compact or weakly compact when it acts between the spaces $BMOA_p$ and $VMOA_p$ for $1<p<2$, and relate this to a Carleson type condition on $g$.

Finally, we show a connection between the mapping properties of specific Volterra-type operators  with
the size of the maximal  subspace of strong continuity $[\phi_t, BMOA_p]$, extending  a related result on BMOA from  \cite[Corollary 2]{bcdmps}.

We will use the notation $C, C', C_1 ...$ to denote constants, the values of which may change from one step to the next. The notation
$X\simeq Y$   means that  the two quantities $X, Y$ are comparable.

\section{Composition Operators}

In this section we discuss boundedness of composition operators induced by univalent symbols on $BMOA_p$ and $VMOA_p$.

\begin{theorem}
Let $1\leq p < 2$, and $\psi : \D\to  \D$ analytic and univalent. Then:
\begin{enumerate}
\item the composition operator $C_\psi:BMOA_p\to BMOA_p$
is bounded, and
\begin{equation}\label{norm-comp}
\n{C_\psi}_{BMOA_p \to BMOA_p}\leq C \left(1+ \log\frac{1+\vert \psi(0)\vert}{1-\vert \psi(0)\vert}\right)\,,
\end{equation}
where $C$ is a constant depending only on $p$.
\item $C_{\psi} :VMOA_p\to VMOA_p$ is also bounded.
\end{enumerate}
\end{theorem}

\begin{proof}
(1) For $f\in BMOA_p$ and $a\in \D$ we write
$$
I(f, p, a)=\int_{\D} \vert (f \circ \vf_a)'(z) \vert^p\,(1-|z|^2)^{p-1}dm(z),
$$
then we have
\[
\begin{split}
I(f\circ \psi,p,a)&=\int_{\D} \vert (f\circ \psi \circ \phi_a)'(z) \vert^p\,(1-\vert z \vert^2)^{p-1}dm(z)
\\
&=\int_{\D} \vert (f\circ\vf_{\psi(a)}\circ\vf_{\psi(a)} \circ \psi \circ \vf_a)'(z) \vert^p\,(1-\vert z \vert^2)^{p-1}dm(z)
\\
&= \int_{\D} |(f\circ \vf_{\psi(a)})'(\sigma_a (z))|^{p} |\sigma'_a (z)|^p\;(1-|z|^2)^{p-1}\,dm(z).
\end{split}
\]
where the function
\[
\sigma_a = \vf_{\psi(a)} \circ \psi \circ \vf_a,
\]
is a self-map of $\D$ and $\sigma_a(0)=0$.
Applying H\"{o}lder's inequality with exponents $2/p$ and $2/(2-p)$  we get
\[
\begin{split}
I(f\circ \psi,p,a)\leq&\left(\int_{\D} |(f\circ \vf_{\psi(a)})'(\sigma_a (z))|^p|\sigma'_a (z)|^2
(1-|z|^2)^{p-1}dm(z)\right)^{\frac{p}{2}}\cdot
\\
& \cdot \left(\int_{\D}|(f\circ \vf_{\psi(a)})'(\sigma_a (z))|^p(1-|z|^2)^{p-1}dm(z)\right)^{\frac{2-p}{2}}\\
&= I_1^{\frac{p}{2}}\cdot I_2^{\frac{2-p}{2}}.
\end{split}
\]
For the first integral $I_1$ we use the inequality $1-|z|^2\leq 1-|\sigma_a(z)|^2$, which is a consequence of
Scwarz's Lemma on $\sigma_a$, and make the change  of variables  $w=\sigma_a(z)$ to obtain
\begin{align*}
I_1 &\leq  \int_{\sigma_a(\D)} | (f\circ \vf_{\psi(a)})'(w)|^p\,(1-|w|^2)^{p-1}\,dm(w)\\
&\leq   \int_{\D} | (f\circ \vf_{\psi(a)})'(w)|^p\,(1-|w|^2)^{p-1}\,dm(w)\\
&= I(f, p, \psi(a)).
\end{align*}
The second integral $I_2$ can be viewed as a weighted area  integral of the composite function  $g\circ\sigma_a$, where
$g=(f\circ\vf_{\psi(a)})'$, against the weight $(1-|z|^2)^{p-1}$. Using known estimates for composition operators on weighted
Bergman spaces (see for example \cite[Lemma 1]{Sisk3}) we obtain
\begin{align*}
I_2&=\int_{\D}|g(\sigma_a (z))|^p(1-|z|^2)^{p-1}\,dm(z)\\
&\leq \left(\frac{\n{\sigma_a}_{\infty}+|\sigma_a(0)|}{\n{\sigma_a}_{\infty}-|\sigma_a(0)|}\right)^{p+1}
\int_{\D}|g(z)|^p(1-|z|^2)^{p-1}\,dm(z)\\
&=\int_{\D} |(f\circ\vf_{\psi(a)})'(z)|^p(1-|z|^2)^{p-1}\,dm(z)\quad (\text{since}\,\, \sigma_a(0)=0)\\
&= I(f, p, \psi(a))
\end{align*}

Putting these together we obtain
$$
I(f\circ \psi,\,p,\,a)\leq I(f,\,p,\,\psi(a))
$$
 for each $a\in \D$,  so
\[
\sup_{a\in\D}I(f\circ \psi,\,p,\,a) \leq \sup_{a\in\D} I(f,\,p,\,\psi(a)) \leq
\sup_{b \in \D}I(f, p, b) \leq  \n{f}^p_{BMOA_p}.
\]
Next using the growth estimate
$$
|f(z)|\leq C_1(1+\log\frac{1+|z|}{1-|z})\n{f}_{BMOA_p}
$$
for functions in $BMOA_p$ which is implicitly proved in \cite{Zhao} we obtain
\begin{align*}
\n{f\circ\psi}_{BMOA_p}^p&=|f(\psi(0))|^p+ \sup_{a\in\D}I(f\circ\psi, p, a)\\
&\leq C_2 \left(1+\log\frac{1+\vert \psi(0)\vert}{1-\vert \psi(0)\vert}\right)^p\norm{f}_{BMOA_p}^p,
\end{align*}
from which (\ref{norm-comp}) follows.

 (2) We are going to use the fact that  polynomials are dense in $VMOA_p$, and also the fact that
$\lim\limits_{r\to 1}\n{f-f_r}_{BMOA_p} =0$  for  each $f\in VMOA_p$, where  $f_r(z)=f(rz)$, $0<r<1$,
are the dilations of $f$. These two properties are in fact equivalent and each of them characterizes
membership of functions in $VMOA_p$, see  \cite[Theorem 3]{cdmv} or \cite[Proposition 2.3]{LP}.

For $f\in VMOA_p$, we need to show that $f\circ\psi\in VMOA_p$. Equivalently we will show that
$$
\lim_{r\to
1}\n{(f\circ\psi)_r-f\circ\psi}_{BMOA_p}=0.
$$
 Note that $(f\circ\psi)_r=f\circ\psi_r$, and  $\psi_r(0)=\psi(0)$ for each $r$. For a
polynomial $P$ we have (dropping the subscript of the norm)
\begin{align*}
\n{(f\circ\psi)_r-f\circ\psi}&=\n{f\circ\psi_r-f\circ\psi}\\
&\leq \n{f\circ\psi_r-P\circ\psi_r}+\n{P\circ\psi_r-P\circ\psi}+\n{P\circ\psi-f\circ\psi}\\
&=\n{(f-P)\circ\psi_r}+\n{P\circ\psi_r-P\circ\psi}+\n{(f-P)\circ\psi}\\
&\leq 2C\left(1+\log\frac{1+|\psi(0)|}{1-|\psi(0)|}\right)\n{f-P}+\n{P\circ\psi_r-P\circ\psi}.
\end{align*}
Now  since $\psi$ is univalent and bounded an easy argument through the integrability of $|\psi'(z)|^2$ over $\D$ gives that
$\psi\in VMOA_p$, and subsequently that $\psi(z)^n\in VMOA_p$ for each positive integer $n$.
Therefore $P\circ\psi\in VMOA_p$ for every polynomial $P$. Given $\e>0$ we can chose a
polynomial $P$ such that $2C(1+\log\frac{1+|\psi(0)|}{1-|\psi(0)|})\n{f-P}< \e/2$. Then we can find $r<1$ such that
$\n{P\circ\psi_r-P\circ\psi}<\e/2$, and the conclusion follows.
\end{proof}

\section{Semigroups of Composition Operators on $BMOA_p$}

First we show that all  semigroups $\{\phi_t\}$ induce  strongly continuous composition semigroups $(T_t)$ on $VMOA_p$.

\begin{theorem}
Let $\{\phi_t\}$ be a semigroup of functions  and $p \geq 1$. Then the induced composition semigroup $(T_t)$ is strongly
continuous   on $VMOA_p$.
\end{theorem}
\begin{proof}
We need to show that if  $f\in VMOA_p$ then,
\[
\lim_{t \to 0^+}\, \n{T_t(f) - f}=0,
\]
(dropping again the subscript of the $BMOA_p$-norm) .
For a  polynomial $P$ we have
\begin{align*}
\n{T_t(f) - f}&\leq \n{T_t(f)-T_t(P)}+ \n{T_t(P)-P}+ \n{P-f}\\
&\leq \n{T_t} \n{f-P}+\n{T_t(P)-P}+ \n{P-f}\\
&\leq (1+\n{T_t})\n{P-f}+\n{T_t(P)-P}.
\end{align*}
  Since the polynomials are dense in $VMOA_p$  and $\sup\limits_{t<1}\n{T_t}<\infty$, it suffices to prove the
  claim for $f=P$, a polynomial. Now the set
of polynomials is contained in the classical Dirichlet space  $\mathcal{D}$ consisting of those  those $f$ for which
\[
\n{f}_{\mathcal{D}}^2=|f(0)|^2+ \int_{\D}| f'(z)|^2\,dm(z)<\infty.
\]
An application of H\"{o}lder's inequality shows that $\mathcal{D}\subset VMOA_p$ and there is a
constant $C=C_p$, such that
\[
\n{f}\leq C\, \n{f}_{\mathcal{D}},\quad \forall f \in \mathcal{D}.
\]
In particular for a polynomial $P$ we have    $P\circ \phi_t - P \in \mathcal{D}$ since $\phi_t$ is univalent, and
\[
\n{T_t(P) - P}\leq C\,\n{T_t(P) - P}_{\mathcal{D}}, \,\, \, t>0.
\]
But composition semigroups  are strongly continuous on $\mathcal{D}$ \cite{Sisk4}, so the last inequality implies
\[
\lim_{t \to 0^+}\, \n{T_t(P) - P}=0
\]
for each  polynomial $P$ and the proof is complete.
\end{proof}

\begin{theorem}
Suppose $\{\phi_t\}$ is a nontrivial semigroup and $1\leq p<\infty$. Then
$[\phi_t, BMOA_p]\subsetneq BMOA_p$.
\end{theorem}

\begin{proof} The result is obtained by observing that the proof of  \cite[Theorem 3.1]{AJS}, whose conclusion is used to
prove  the analogous result \cite[Theorem 1.1]{AJS}, in fact applies here.
More precisely,  in the proof of   \cite[Theorem 3.1]{AJS}  the authors use appropriate test functions $f\in H^\infty$  which
satisfy
$$
\liminf_{t\to 0}\n{f\circ\phi_t - f}_{\mathcal{B}}\geq \delta>0,
$$
for each nontrivial $\{\phi_t\}$. Those test functions $f$ are infinite interpolating Blaschke products whose
zeros lie on a radius  $\{r\g_0:  0<r<1\}$ for some   $\g_0\in\partial\D$.
 On the other hand  from \cite[Theorem 3.12]{GPV} (see also \cite{DM}), we know that for such Blaschke products $B(z)$,
 $$
 B(z) \in \bigcap_{s\in{(0,1)}}Q_s.
 $$
 Since $\bigcup\limits_{0<s<1}Q_s\subset BMOA_p$ for all $p>0$,  these test functions can be used
 in our context, and the proof can be completed by following  the lines of the proof in \cite{AJS}.
\end{proof}

\begin{remark} From the above it follows that  the hypothesis $H^\infty \subset X\subset \mathcal{B}$
 in the theorem of Anderson, Jovovic and Smith \cite[Theorem 1.1]{AJS} can be replaced by
 \[
 \bigcap_{s\in{(0,1)}} Q_s\subset X\subset \mathcal{B} 
 \]
or by some similar condition which will imply that interpolating Blaschke products belong to $X$.
\end{remark}

\begin{theorem}\label{theNec}
Let $1\leq p<2\,$ and let $\{\phi_t\}$ be a semigroup with infinitesimal generator $G$. If
\begin{equation}\label{pLog}
\lim_{| I|\to 0} \frac{\left(\log\frac{2}{\vert I \vert}\right)^p}{\vert I \vert}
\int_{S(I)} \frac{(1-\vert z\vert^2)^{p-1}}{\vert G(z) \vert^p}\,dm(z)=0,
\end{equation}
then $VMOA_p = [\phi_t, BMOA_p]$.
\end{theorem}

\begin{proof} Using the description  (\ref{ssc}) of the maximal subspace of strong continuity,
it suffices to show that if $h\in BMOA_p$ is such that $Gh'\in BMOA_p$ then $h\in VMOA_p$. Set
$f=Gh'$ for such an $h$ and select  $r$  sufficiently close to $1$ so that $ G$ has no zeros in the ring
$D_r=\left\{z\in\D: r<\abs{z}<1\right\}$. Then    $1/G$ is analytic on $D_r$.
 Let   $I\subset \partial \D$ be an arc  sufficiently small so that $S(I)\subset D_r$, write
$z_I=(1-|I|)\xi$ where $\xi \in \partial\D$ is the center of $I$, and set
$$
d\mu(z)=\frac{(1-|z|^2)^{p-1}}{|G(z)|^p}\,dm(z).
$$
Then $h'=f/G$ and we have
\begin{align*}
\frac{1}{|I|}&\int_{S(I)}|h'(z)|^p(1-|z|^2)^{p-1}\,dm(z) =\frac{1}{|I|}\int_{S(I)}|f(z)|^p\,d\mu(z)\\
&\leq 2^p\frac{1}{|I|}\int_{S(I)}|f(z)-f(z_I)|^p\,d\mu(z)+ 2^p\frac{|f(z_I)|^p}{|I|}\int_{S(I)}\,d\mu(z)\\
&\leq 2^p\frac{1}{|I|}\int_{S(I)}\left|\frac{f(z)-f(z_I)}{1-\bar{z}_I z}\right|^p
|1-\bar{z}_I z|^p \,d\mu(z)+2^p C^p \frac{(\log\frac{2}{|I|})^p}{|I|}\mu(S(I))\\
&= 2^pA_I + 2^pCB_I
\end{align*}
where we have used the inequality $(x+y)^p\leq 2^p(x^p+y^p)$ and the growth estimate
$|g(z)|\leq C\log\frac{2}{1-|z|}$ for functions
$g\in BMOA_p$. By hypothesis (\ref{pLog}),
$$
B_I= \frac{(\log\frac{2}{|I|})^p}{|I|}\mu(S(I)) \to 0, \quad \text{as}\, |I|\to 0.
$$
We will show that the same is true for $A_I$. Using the  estimate
$|1-\bar{z}_I z|\simeq |I|$  for $ z\in S(I)$, and applying H\"{o}lder's inequality we have
\begin{align*}
A_I&\leq C|I|^{p-1}\int_{S(I)}\left|\frac{f(z)-f(z_I)}{1-\bar{z}_I z}\right|^p\, d\mu(z)\\
 &\leq C|I|^{p-1}\left(\int_{S(I)}\left|\frac{f(z)-f(z_I)}{1-\bar{z}_I z}\right|^2 \,d\mu(z)\right)^{\frac{p}{2}}
\left(\int_{S(I)}d\mu(z)\right)^{\frac{2-p}{2}}\\
&= C|I|^{p-1}\mu(S(I))^{\frac{2-p}{2}}
\left(\int_{\D}\left|\frac{f(z)-f(z_I)}{1-\bar{z}_I z}\right|^2 \,d\mu(z)\right)^{\frac{p}{2}}.
\end{align*}

The hypothesis (\ref{pLog}) implies that for  $|I|$  sufficiently small
$\mu(S(I))  \leq  |I|\left(\log\frac{e}{| I |}\right)^{-p}$, and that  $\mu$ is a Carleson measure,
 i.e. (\ref{CarH2}) holds.  Thus we have

\begin{align*}
A_I &\leq C|I|^{p-1}\left( |I|(\log\frac{2}{| I |})^{-p} \right)^{\frac{2-p}{2}}
\left(\int_{\D}\left|\frac{f(z)-f(z_I)}{1-\bar{z}_I z}\right|^2 \,d\mu(z)\right)^{\frac{p}{2}}\\
&\leq C C_{\mu}^{p/2}\frac{|I|^{\frac{p}{2}}}{(\log\frac{2}{|I|})^{\frac{2-p}{2}p}}
\left( \int_{\partial\D}\left|\frac{f(\zeta)-f(z_I)}{1-\bar{z}_I \zeta}\right|^2
 |d\zeta|\right)^{\frac{p}{2}}\\
&\simeq \frac{1}{(\log\frac{2}{\abs{I}})^{\frac{2-p}{2}p}}
 \left((1-\abs{z_I}) \int_{\partial\D}\left\vert\frac{f(\zeta)-f(z_I)}{1-\bar{z}_I \zeta}\right\vert^2
 \abs{d\zeta}\right)^{\frac{p}{2}}\\
 &\leq \frac{1}{(\log\frac{2}{\abs{I}})^{\frac{2-p}{2}p}}\left(\sup_{a\in \D}(1-|a|^2)
 \int_{\partial\D}\left|\frac{f(\zeta)-f(a))}{1-\bar{a} \zeta}\right|^2
 \abs{d\zeta}\right)^{\frac{p}{2}}.
\end{align*}
Now the quantity inside the parenthesis is the square of the Garsia norm   of $f\in BMOA_p\subset BMOA$.
Since the Garsia norm is comparable to the BMOA norm (\cite{Girela}), i.e.
\begin{equation}\label{Garcia}
\n{f}^2_{BMOA}\simeq \sup_{a \in \D} \int_{\partial\D}
|f(\zeta) - f(a)|^2\frac{1-|a|^2}{|1-\bar{a}\zeta|^2}|d\zeta|,
\end{equation}
we get
$$
A_I\leq\frac{1}{(\log\frac{2}{\abs{I}})^{\frac{2-p}{2}\,p}} \n{f}^p_{BMOA}\to 0 \quad \text{as}\, |I|\to 0.
$$
It follows that
\[
\lim_{|I|\to 0}\frac{1}{|I|}\int_{S(I)}|h'(z)|^p(1-|z|^2)^{p-1}\,dm(z) =0
\]
which means that $h$ is in $VMOA_p$ and the proof is complete.
\end{proof}

The next  Corollary says that there are plenty of  semigroups for which $[\phi_t, BMOA_p]=VMOA_p $.

\begin{corollary}
Let $\,1\leq p <2\,$ and $\,a\in(0,1)\,$. If
\begin{equation}\label{cond2}
\frac{(1-\abs{z})^a}{G(z)}\,=\,O(1),\;\;\abs{z}\to 1\, ,
\end{equation}
then $ [\phi_t, BMOA_p]=VMOA_p$.
\end{corollary}
\begin{proof}
The hypothesis  (\ref{cond2}) implies
$$
\frac{(1-|z|^2)^{p-1}}{|G(z)|^p}\leq C(1-|z|)^{p(1-a)-1}, \quad |z|\to 1,
$$
therefore
\begin{align*}
\int_{S(I)} \frac{(1-| z|^2)^{p-1}}{| G(z)|^p}\,dm(z)&\leq C\int_{S(I)}(1-|z|)^{p(1-a)-1}\,dm(z)\\
&\leq C|I|\int_{1-|I|}^1(1-r)^{p(1-a)-1}\,dr\\
&=C\frac{|I|^{p(1-a)+1}}{p(1-a)}.
\end{align*}
Now since
$$
 \left(\log\frac{2}{\vert I \vert}\right)^p\,\frac{\abs{I}^{p(1-a)}}{p(1-a)}\longrightarrow  0 \quad \text{as}\,|I|\to 0,
$$
the condition  (\ref{pLog})   is satisfied     and the conclusion follows.
\end{proof}

Next we prove a  necessary condition for $ [\phi_t, BMOA_p]=VMOA_p$, for semigroups with Denjoy-Wolff point in the open disc.

\begin{theorem}
Let $\{\phi_t\}_{t\geq 0}$ be a semigroup with infinitesimal generator $G$ and Denjoy-Wolff
point $b\in\D$. If $VMOA_p=[\phi_t, BMOA_p]$, then
\begin{equation}\label{cond3}
\lim_{\abs{I}\to 0} \frac{1}{\abs{I}}\int_{S(I)} \frac{(1-\vert z\vert^2)^{p-1}}{\vert G(z) \vert^p}\,dm(z)=0.
\end{equation}
\end{theorem}
\begin{proof}
We assume, without loss of generality, that $b=0$. Then, from (\ref{D-W})
\[
G(z)=-zP(z),\quad Re(P)\geq 0.
\]
Next, we consider the function
\[
g(z)=\int_0^z \frac{u}{G(u)}\,du=-\int_0^z \frac{1}{P(u)}\,du,
\]
Since $Re(\frac{1}{P})\geq 0$  the function $g$ is univalent \cite[Proposition 1.10]{Pomm}, and from the growth estimate
 $|1/P(z)|=O(\frac{1}{1-|z|}),\, z\in \D$, for functions of non-negative real part it follows that $(1-|z|)|g'(z)|$ is bounded on $\D$, so
 $g\in \mathcal{B}$.  But univalent functions of the Bloch space are the same as those in $BMOA_p$ \cite[p. 134]{PerRat2},
 hence $g\in BMOA_p$.

In addition $G(z)g'(z)=z\in BMOA_p$.
This means that $g\in[\phi_t, BMOA_p]\,$ and by hypothesis $g\in VMOA_p$. That is
\[
\lim_{\abs{I}\to 0} \,\frac{1}{\abs{I}}\,\int_{S(I)} \frac{\abs{z}^p}{\vert G(z) \vert^p}(1-\vert z\vert^2)^{p-1}\,dm(z)=0
\]
and (\ref{cond3}) follows since $|z|\simeq 1$ for $z\in S(I)$ when $I$ is small.
\end{proof}

Note that there is a gap between the conditions  (\ref{pLog}) and (\ref{cond3}), and finding a characterization
of the equality $VMOA_p = [\phi_t, BMOA_p]$ in terms of the generator $G$  seems to be difficult.
Such a characterization however can be given,  following \cite{bcdmps}, in terms of the  resolvent operator
$$
\mathcal R (\lambda, \Gamma)=(\lambda-\Gamma)^{-1}
$$
of the infinitesimal generator $\Gamma(f)=Gf'$ of
the semigroup $(T_t)$ acting on $VMOA_p$ when  $1<p<2$. For those values of $p$ it follows from the work of
K.M. Perfekt \cite{Perf1} that we have the duality
\begin{equation}\label{duality}
VMOA_p^{**}=BMOA_p.
\end{equation}
Indeed Perfekt shows that  if $X$ is a reflexive and separable Banach space, then its M\"{o}bius invariant subspace
\[
M(X)=\{f\in X: \;\sup_{a\in \D}\norm{f\circ \phi_a - f}_X<\infty\}
\]
and the  little-o version
\[
M_0(X)=\{f\in X: \;\lim_{\abs{a}\to1 }\norm{f\circ \phi_a - f}_X=0\}
\]
have the duality property
\[
M_0(X)^{**}=M(X).
\]
More details on the results from \cite{Perf1} will be given in Section 5.  If  $X=D^p_{p-1}$,  $1<p<2$,  then  $M(X)=BMOA_p$
and $M_0(X)=VMOA_p$ and $D^p_{p-1}$ is reflexive, as it can be seen by adjusting Luecking's result on the dual spaces of
weighted Bergman spaces \cite[Theorm 2.1]{Lueck}. Thus we get the duality (\ref{duality}).

A characterization of the semigroups for which $[\phi_t, BMOA]=VMOA$
was given in \cite{bcdmps} in terms of properties of the resolvent operator. Following the same steps  the result
is seen to hold for $BMOA_p$ for all $1<p<2$ as the  following proposition describes.

 \begin{proposition}
Let $\{\phi_t\}$ be a semigroup and  let $\Gamma$ be the infinitesimal generator of the composition semigroup $(T_t)$ acting on
$VMOA_p$, $1<p<2$. Let
$\rho(\Gamma)=\{\lambda\in \C: \mathcal R (\lambda, \Gamma)=(\lambda-\Gamma)^{-1}:VMOA_p\to VMOA_p\,\, \, \text{is bounded}\}$,
the resolvent set. Then for  $\lambda\in \rho(\Gamma)$ the following statements are equivalent:
\begin{enumerate}
\item $[\phi_t, BMOA_p]=VMOA_p$\,.
 \item $\mathcal R(\lambda, \Gamma)^{**}(BMOA_p) \subset VMOA_p$.
 \item $\mathcal R(\lambda, \Gamma)$ is weakly compact on $VMOA_p$.
 \end{enumerate}
 \end{proposition}

As mentioned above the proof follows the lines of \cite[Theorem 4]{bcdmps} and we omit the details.

\section{Volterra-type Operators on $BMOA_p$}

In this last  section we study  Volterra-type operators $T_g$ and  point out a   connection between
the maximal subspace  $[\phi_t, BMOA_p]$   and some properties of these operators. The operators $T_g$ are
defined by
\[
T_g(f)(z)=\int_0^z f(\zeta) g'(\zeta)\,d\zeta\, \quad f\in H(\D),
\]
where  $g\in H(\D)$ is the inducing symbol.
They were first considered by Ch. Pommerenke \cite{Pom}. He proved that $T_g$ is bounded
on the Hardy space $H^2$ if and only if $g\in BMOA$. This characterization of boundedness was
extended to all Hardy spaces $H^p$ by
A. Aleman and A. Siskakis \cite{Al.Sisk 1} and the operator was subsequently studied on several
spaces of analytic functions by many authors.

On the spaces of our concern, $T_g$  were studied by J. Pau and R. Zhao \cite{PaZh}. They considered  $T_g$ on
  the more general family of  the spaces $F(p,q,s)$  and in particular they proved for $BMOA_p=F(p, p-2, 1)$
  and for $1\leq p<2$ that  $T_g : BMOA_p\to BMOA_p$ is bounded if and only if $g$ satisfies
    \begin{equation}\label{logCar}
\sup_{ I \subset \partial \D}\left(\frac{\left(\log\frac{2}{\vert I \vert}\right)^p}{\vert I \vert}
 \int_{S(I)} \vert g'(z) \vert^p\, (1-\vert z\vert^2)^{p-1}\,dm(z)\right)<\infty
 \end{equation}
  and is compact if and only if the corresponding little-oh condition is valid, i.e. the
  quantity inside the parenthesis goes to $0$ as $|I|\to 0$.

These results were extended by C. Yuan and C. Tong who proved that for $1<p<2$, the operator
$T_g: BMOA \to BMOA_p$ is bounded if and only if the same above condition (\ref{logCar}) holds
for $g$, and is compact if and only if
the corresponding little-oh condition is valid
\cite [Theorem 14 \& Corollary 16, (3)]{YuTo}

In the next theorem we show that for the missing case $p=1$ in the results of \cite{YuTo}, an analogous  characterization
of boundedness and compactness holds.

\begin{theorem}
Let $  g$ be an analytic function on $\D$. Then:
\begin{enumerate}
\item The operator $T_g : BMOA \to BMOA_1 $ is bounded if and only if
 \begin{equation}\label{logCar1}
\sup_{I \subset \partial\D}\left( \frac{\log\frac{2}{|I|}}{| I|} \int_{S(I)} | g'(z)|\,dm(z)\right)
<\infty.
 \end{equation}
 \item The following  are equivalent:
 \begin{enumerate}[(i)]
\item $T_g: BMOA \to BMOA_1$ is compact.
\item The function $g$ satisfies,
\begin{equation}\label{logCar1-0}
\lim_{|I|\to 0}\left(\frac{\log\frac{2}{| I |}}{| I|}
\int_{S(I)} | g'(z)|\,dm(z)\right)=0.
\end{equation}
\item $T_g(BMOA) \subset VMOA_1$.
\end{enumerate}
\end{enumerate}
\end{theorem}

\begin{proof}
(1) Suppose  (\ref{logCar1}) holds and let $f\in BMOA$. Let $z_I=(1-|I|)\xi$, where $\xi \in \partial\D$ is the center of $I$. Then
\begin{equation}\label{a1}
\n{T_g(f)}_{BMOA_1}\simeq \sup_{I\subset \partial\D}\frac{1}{|I|}\int_{S(I)}|f(z)g'(z)|\,dm(z).
\end{equation}
Setting
$$
d\mu(z)=d\mu_g(z):= |g'(z)|dm(z),
$$
(this notation will be used throughout the proof) we have
\begin{align*}
\frac{1}{|I|}&\int_{S(I)}|f(z)g'(z)|\,dm(z)=\frac{1}{|I|}\int_{S(I)}|f(z)|\,d\mu(z)\\
&\leq   \frac{1}{|I|}\int_{S(I)}|f(z)-f(z_I)|\,d\mu(z)+
|f(z_I)|\frac{1}{|I|}\int_{S(I)}|g'(z)|\,dm(z)\\
&\leq \frac{1}{|I|}\int_{S(I)}|f(z)-f(z_I)|\,d\mu(z)+C\left(\frac{\log\frac{2}{|I|}}{|I|}\int_{S(I)}|g'(z)|\,dm(z)\right)\n{f}_{BMOA}\\
&\leq A_I(f)+C'\n{f}_{BMOA},
\end{align*}
where we have used the growth estimate for  $BMOA$ functions and the hypothesis (\ref{logCar1}).
Using the estimate $|1-\bar{z}_I z|\simeq |I|$ for $z\in S(I)$, the quantity $ A_I(f)$ is
\begin{align*}
A_I(f) &=\frac{1}{|I|} \int_{S(I)}\left|\frac{f(z)-f(z_I)}{1-\bar{z}_I z}\right||1-\bar{z}_I z|\,d\mu(z)\\
&\leq  C\int_{S(I)}\left|\frac{f(z)-f(z_I)}{1-\bar{z}_I z}\right|\,d\mu(z)
\end{align*}
 We can then continue with the last  integral
by repeating the steps in the proof of Theorem \ref{theNec}. In specific,  apply a
Cauchy-Schwarz inequality, use (\ref{logCar1}) and its implication that   $d\mu(z)$ is a Carleson measure, and change to polar coordinates  to obtain
$$
\int_{S(I)}\left|\frac{f(z)-f(z_I)}{1-\bar{z}_I z}\right|\,d\mu(z)\leq \frac{C}{(\log\frac{2}{|I|})^{1/2}}\n{f}_G
$$
where $\n{f}_G$ is the Garsia norm of $f$, which is equivalent to $\n{f}_{BMOA}$. It follows then from (\ref{a1}) that
$\n{T_g(f)}_{BMOA_1}\leq C\n{f}_{BMOA}$ for some constant $C$ and for all $f\in BMOA$,
i.e. $T_g:BMOA \to BMOA_1$ is bounded.

 For the converse, assume that $T_g:BMOA \to BMOA_1$ is bounded and use the test functions
 $$
 f_a(z)=\log\frac{1}{1-\bar{a}z}, \quad a\in \D,
 $$
 which form a bounded set in $BMOA$. If $a=z_I$ then $|f_a(z)|\simeq \log\frac{2}{|I|}$ for $z\in S(I)$ and so
\begin{align*}
\frac{\log\frac{2}{|I|}}{|I|}\int_{S(I)}|g'(z)|\,dm(z)&\leq C\frac{1}{|I|}\int_{S(I)}|f_a(z)||g'(z)|\,dm(z)\\
&\leq C'\n{T_g(f_a)}_{BMOA_1}\\
&\leq C' \n{T_g}\n{f_a}_{BMOA}\\
&\leq C''\n{T_g}
\end{align*}
so (\ref{logCar1}) holds.

(2) We first show that (i) and (ii) are equivalent. First assume (ii) holds, then we will show that
$T_g : BMOA\to BMOA_1$ is compact. Using a well known characterization of compactness of operators
on function spaces, see for example \cite[Lemma 3.7]{Tjani},  it suffices to show that if  $(f_n)$ is a
bounded sequence in BMOA such that $f_n\to 0$ uniformly on each compact subset of $\D$
then $\n{T_g(f_n)}_{BMOA_1}\to 0$. Equivalently we want to show that
\begin{equation}\label{compact}
\sup_{I}\frac{1}{\abs{I}}\int_{S(I)}\abs{f_n(z)}\abs{g'(z)}\,dm(z) \to 0 \quad \text{as}\,\,n\to\infty.
\end{equation}
To show this we write
\begin{align*}
\frac{1}{|I|}\int_{S(I)}|f_n(z)| |g'(z)|\,dm(z)&\leq
 \frac{1}{|I|}\int_{S(I)}|f_n(z)-f_n(z_I)| |g'(z)|\,dm(z)\\
  &+|f_n(z_I)|\frac{1}{|I|}\int_{S(I)}|g'(z)|\,dm(z)\\
&=A_n(I)+B_n(I),
\end{align*}
where $z_I$ is as before.

We first treat $B_n(I)$. Using  the assumptions on $(f_n)$ we have
$$
|f_n(z_I)|\leq C\n{f_n}_{BMOA}\log\frac{2}{1-|z_I|}\leq C'\log\frac{2}{|I|},
$$
for all $I$ and $n$, so that
$$
B_n(I)\leq C'\frac{\log\frac{2}{|I|}}{\abs{I}}\int_{S(I)}\abs{g'(z)}\,dm(z)
$$
for all $n$. Now given  $\e>0$, by the  hypothesis (\ref{logCar1-0}) and the above inequality, there is
a $\d>0$ such that if $|I|< \d$ then $B_n(I)<\e/2$ for all $n$, thus,
$ \sup_{|I|<\d}B_n(I)\leq \e/2$ for all $n$.

On the other hand if $I$ is an arc with  $|I|\geq \d$ then  $|z_I|=1-|I|\leq 1-\d$, so $z_I$ is inside the
closed disc $\{|z|\leq  1-\d\}$. This and  the assumption  of the convergence of $(f_n)$  to $0$, uniformly
on compact sets, imply that there is an $N_0$ such that for $n\geq N_0$
$$
|f_n(z_I)|\frac{1}{|I|}\int_{S(I)}|g'(z)|\,dm(z) \leq
 (\sup_{|z|\leq  1-\d}|f_n(z)|)\n{g}_{BMOA_1} \leq \e/2,
$$
thus $\sup\limits_{|I|\geq \d}B_n(I)\leq \e/2$, for $n\geq N_0$.
 It follows then that
\begin{align*}
\sup_{I}B_n(I)=\max\big\{\sup_{|I|< \d}B_n(I),
\, \sup_{|I|\geq \d}B_n(I) \big \}
\leq \sup_{|I|<\d}B_n(I)+ \sup_{|I|\geq  \d}B_n(I)
\leq \e,
\end{align*}
for $n\geq N_0$, therefore
\begin{equation}\label{comp b}
\lim_{n\to\infty}\sup_{I}B_n(I)=\lim_{n\to\infty}\left(\sup_{I}|f_n(z_I)|\frac{1}{|I|}\int_{S(I)}|g'(z)|\,dm(z)\right)=0.
\end{equation}

Now we consider  $A_n(I)$. Using the estimate $|1-\bar{z_I}z | \simeq |I|$ for $z\in S(I)$
and recalling that   $d\mu(z)=|g'(z)|dm(z)$  we have
\begin{equation}\label{e1}
A_n(I) \simeq
\int_{S(I)} \frac{\abs{f_n(z)-f_n(z_I)}}{\vert 1-\bar{z_I}z\vert}\,d\mu(z).
\end{equation}
The  hypothesis (\ref{logCar1-0})  implies
$$
\mu{(S(I))} \leq C \frac{|I|}{\log\frac{2}{|I|}} \quad \text{for all}\,\, I,
$$
 and applying the  Cauchy-Schwartz inequality we obtain
\begin{align*}
A_n(I)^2&\simeq \left(\int_{S(I)} \frac{\abs{f_n(z)-f_n(z_I)}}{\vert 1-\bar{z_I}z\vert}\,d\mu(z)\right)^2\\
&\leq \mu(S(I))  \int_{S(I)}\Big\vert \frac{f_n(z)-f_n(z_I)}{ 1-\bar{z_I}z} \Big\vert^2\,d\mu(z)\\
&\leq C\frac{\abs{I}}{\log\frac{2}{\abs{I}}}
 \int_{\D}\Big\vert \frac{f_n(z)-f_n(z_I)}{ 1-\bar{z_I}z} \Big\vert^2\, d\mu(z)\\
&\simeq\frac{\abs{I}}{\log\frac{2}{\abs{I}}} \int_{\D\setminus D_r}
\Big\vert \frac{f_n(z)-f_n(z_I)}{ 1-\bar{z_I}z} \Big\vert^2\, d\mu(z)
+\frac{\abs{I}}{\log\frac{2}{\abs{I}}} \int_{D_r}
\Big\vert \frac{f_n(z)-f_n(z_I)}{ 1-\bar{z_I}z} \Big\vert^2\, d\mu(z),
\end{align*}
where $D_r=\{z\in \D:\,\abs{z}\leq r\}$ and  $r\in (0,1)$ will  be chosen later.
We consider each term of this sum separately.

For the integral in the first term, putting  $d\mu_r=d\mu|_{\D\setminus D_r}$ and applying (\ref{CarH2}) in the first inequality, we have
\begin{align*}
\int_{\D\setminus D_r}& \left| \frac{f_n(z)-f_n(z_I)}{ 1-\bar{z_I}z}
\right|^2\, d\mu(z) = \int_{\D}
 \left | \frac{f_n(z)-f_n(z_I)}{ 1-\bar{z_I}z} \right |^2\, d\mu_r(z)\\
&\leq C \n{\mu_r}_{CM}\int_{\partial\D} \left | \frac{f_n(\zeta)-f_n(z_I)}{ 1-\bar{z_I}\zeta}
\right|^2\,\abs{d\zeta}\quad \\
&=C \n{\mu_r}_{CM}\frac{1-|z_I|}{|I|} \int_{\partial\D} \left| \frac{f_n(\zeta)-f_n(z_I)}{ 1-\bar{z_I}\zeta}
\right |^2\,\abs{d\zeta}\\
&\leq C'\n{\mu_r}_{CM}\frac{1}{|I|}\sup_{a\in \D}\left((1-|a|^2)\int_{\partial\D}
\Big\vert \frac{f_n(\zeta)-f_n(a)}{ 1-\bar{a}\zeta}
\Big\vert^2\,\abs{d\zeta}\right)\\
&\leq C'' \frac{1}{|I|}\n{\mu_r}_{CM} \n{f_n}_{BMOA}^2\\
&\leq C''' \frac{1}{|I|}\n{\mu_r}_{CM},
\end{align*}
for all $I$.  At this point, since $\mu$ is a vanishing Carleson
measure, we can use   \cite[Lemma 15]{YuTo} which says that $\lim\limits_{r\to 1^{-}}\n{\mu_r}_{CM}=0$.
 Thus given $\e >0$ we can find an  $r$ such that
\begin{equation}\label{an1}
\sup_{I}\,\frac{\abs{I}}{\log\frac{2}{\abs{I}}} \, \int_{\D\setminus D_r}
\Big\vert \frac{f_n(z)-f_n(z_I)}{ 1-\bar{z_I}z} \Big\vert^2\, d\mu(z)
< \e ,\quad \text{for all}\,\, n.
\end{equation}

Now we consider  the  second term, with $r$ fixed as above. For $|z|\leq r$ we have
$|1-\bar{z}_Iz|^2>(1-r)^2$ and $\sup\limits_{|z|\leq r}|g'(z)|<\infty$,   so there are constants such that
\begin{align*}
\int_{D_r}
\left |\frac{f_n(z)-f_n(z_I)}{ 1-\bar{z_I}z} \right|^2\, d\mu(z)&\leq C_r\int_{D_r}
 | f_n(z)|^2 \, d\mu(z)+ C_r |f_n(z_I)|^2 \int_{D_r} \, d\mu(z)\\
  &\leq C_r' \sup_{|z|\leq r}|f_n(z)|^2+C_r'|f_n(z_I)|^2,
\end{align*}
and we then have
\begin{align*}
\frac{|I|}{\log\frac{2}{|I|}} \int_{D_r}
\left | \frac{f_n(z)-f_n(z_I)}{ 1-\bar{z_I}z} \right|^2\, d\mu(z)
&\leq C_r'\frac{|I|}{\log\frac{2}{|I|}}\sup_{|z|\leq r}|f_n(z)|^2+
C_r'\frac{|I|}{\log\frac{2}{|I|}}|f_n(z_I)|^2\\
&\leq C_r''\sup_{|z|\leq r}|f_n(z)|^2+
C_r'\frac{|I|}{\log\frac{2}{|I|}}|f_n(z_I)|^2.
\end{align*}
By the  hypothesis on $(f_n)$, the first term above tends to $0$ as $n\to \infty$.
  In addition for each small positive $\d$ we have,
\begin{align*}
\sup_{I}\frac{|I|}{\log\frac{2}{|I|}}|f_n(z_I)|^2 &\leq \sup_{|I|<\d}\frac{|I|}{\log\frac{2}{|I|}}|f_n(z_I)|^2
+\sup_{|I|\geq \d}\frac{|I|}{\log\frac{2}{|I|}}|f_n(z_I)|^2\\
&\leq C \sup_{|I|<\d}\frac{|I|}{\log\frac{2}{|I|}}\left(\log\frac{2}{|I|}\right)^2+
\sup_{|I|\geq \d}\frac{|I|}{\log\frac{2}{|I|}}|f_n(z_I)|^2 \\
&=C\sup_{|I|<\d}|I|\log\frac{2}{|I|}+\sup_{|I|\geq \d}\frac{|I|}{\log\frac{2}{|I|}}|f_n(z_I)|^2.
\end{align*}
Now for the given $\e$ we can choose  $\d$ such that if $|I|<\d$, then ${C|I|\log\frac{2}{|I|}<\e}$, so that
$$
\sup_{|I|<\d}\frac{|I|}{\log\frac{2}{|I|}}|f_n(z_I)|^2<\e,
$$
 for all $n$. At the same time if $|I|\geq \d$ then $|z_I|\leq 1-\d$ and by the hypothesis on $(f_n)$
and the fact that $ \frac{|I|}{\log\frac{2}{|I|}}$ remains bounded,  we can find $N_1$ such that
$$
\sup_{|I|\geq \d}\frac{|I|}{\log\frac{2}{|I|}}|f_n(z_I)|^2<\e
$$
for $n\geq N_1$. It follows from the last two inequalities that
$$
\lim_{n\to\infty}  \sup_{I}\frac{|I|}{\log\frac{2}{|I|}}|f_n(z_I)|^2=0.
$$
Putting the above together gives
\begin{equation}
\lim_{n \to \infty}\left(\sup_{I} \frac{\abs{I}}{\log\frac{2}{\abs{I}}} \, \int_{D_r}
\Big\vert \frac{f_n(z)-f_n(z_I)}{ 1-\bar{z_I}z} \Big\vert^2\, d\mu(z)\right) = 0.
\end{equation}
From this  and (\ref{an1}) we obtain $\lim\limits_{n\to\infty}\sup_{I}A_n(I)=0$, and finally in combination with
(\ref{comp b}) the desired conclusion (\ref{compact}) is obtained.

To show the converse assume  $T_g: BMOA \to BMOA_1$ is compact. Then, taking into account that there is a constant $C$ such that
for each $f\in BMOA_1$, $\n{f}_{BMOA}\leq C\n{f}_{BMOA_1}$, we can easily verify that the restriction
$$
T_g{|_{BMOA_1}}: BMOA_1\to BMOA_1
$$
is a compact operator. It then follows from \cite[Theorem 5.2 (ii)]{PaZh}, applied for  the space $BMOA_1=F(1, -1, 1)$,  that
$$
\lim_{|a|\to 1}\log\frac{2}{1-|a|^2}\int_{\D}|g'(z)|(1-|\phi_a(z)|^2)\,dm(z) =0
$$
and this condition on $g$ is well known to be equivalent to  (\ref{logCar1-0}).

Next we show that conditions (ii) and (iii) are equivalent. Suppose (ii) holds and  $f\in BMOA$, we need to show that
$T_g(f)\in VMOA_1$ or equivalently that
$$
\lim_{|I|\to 0} \frac{1}{|I|}\int_{S(I)}|f(z)| |g'(z)|\,dm(z) = 0.
$$
To do this we can  follow the steps in the proof of Theorem \ref{theNec} for the value  $p=1$,
setting $d\mu(z)=|g'(z)|dm(z)$. As in that proof the calculations  lead to an inequality involving the Garsia norm of $f$
and finally to the conclusion that the above limit is $0$, so condition (\ref{logCar1-0}) implies that $T_g(BMOA)\subset VMOA_1$.

 To show the converse we use the test functions
\[
f_I(z)=\log\frac{1}{1-\bar{z}_I z}\in BMOA,
\]
thus $T_g(f_I)\in VMOA_1$. Using  the standard estimate $|f_I(z)|\simeq \log\frac{2}{|I|}$ for $z\in S(I)$ we have
\begin{align*}
\frac{\log\frac{2}{|I||}}{|I|}\int_{S(I)}|g'(z)|\,dm(z)&\leq \frac{C}{|I|} \int_{S(I)}|f_I(z)||g'(z)|\,dm(z)\\
&=\frac{C}{|I|}\int_{S(I)}|T_g(f_I)'(z)|\,dm(z)
\end{align*}
which gives the desired conclusion.
\end{proof}

To continue the study of compactness and weak compactness of Volterra operators on $BMOA_p$ we will need some
results of K.M. Perfekt from \cite{Perf1},\cite{Perf2}. In these articles there is a general construction of pairs of spaces $(M_0, M)$
obtained by little oh and big oh conditions respectively. More precisely let
$X, Y$ be two Banach spaces, where $X$ is reflexive and separable,
and let  $\mathcal{L}$ be  a family of bounded linear operators $L: X\to Y$, equipped with a $\sigma$-compact,
locally compact Hausdorff
topology $\tau$, such that for every $x\in X$ the map $T_x:\mathcal{L}\to Y$, $T_x(L)=L(x)$ is continuous.
A pair of spaces is then defined
\[
\begin{split}
M&=M(X,\mathcal{L})= \{x\in X: \sup_{L\in\mathcal{L}}\n{L(x)}_Y<\infty\},
\\
M_0&=M_0(X,\mathcal{L})= \{x\in M: \limsup_{L\to\infty}\n{L(x)}_Y=0\}
\end{split}
\]
where $L\to\infty$ in the definition of $M_0$ means that $L$ eventually escapes all compact subsets of the topological space
$(\mathcal{L}, \tau)$. Under appropriate  assumptions on $\mathcal{L}$ the quantity
  $\n{x}_M=\sup_{L\in \mathcal{L}}\n{L(x)}_Y$ is a norm on $M$ and, by replacing
$X$ by the closure of $M$ in $X$ and repeating the construction,  it may be assumed that $M$ is dense in $X$.
It is clear that $M_0$ is a closed subspace of $M$. This construction puts in a very general frame
 classical pairs such as  the mean oscillation pair $(VMOA, BMOA)$,
the Bloch pair $(\mathcal{B}_0, \mathcal{B})$, other M\"{o}bius invariant pairs,
Lipschitz-H\"{o}lder spaces, etc. Using  functional
analysis methods it is then proved under mild additional assumptions that the duality
$M_0^{**}=M$ holds in this general context
and also the following theorem on weak compactness.

\begin{thmx}(\cite[Theorem 3.2]{Perf2}). \textit{Let $Z$ a Banach space. A bounded linear operator $T: M_0(X, \mathcal{L})\to Z$
is weakly compact if and only if for each $\e>0$, there is a $C>0$ such that
$$
\n{T(x)}_Z\leq C\n{x}_X+ \e\n{x}_M, \quad \text{for each}\,\,x\in M_0(X, \mathcal{L})
$$
}
\end{thmx}

\vspace{0.2cm}
In addition as a corollary, the same above characterization holds for weak compactness of  operators
$T: M(X, \mathcal{L})\to Z$ provided that they are $weak^{*}-weak$ continuous with respect to the
duality $M_0^{**}=M$.

To obtain the pair of spaces $(VMOA_p,\, BMOA_p)$, we apply the above construction
with $X=Y=D^p_{p-1}$,   $ 1<p< 2$, and
$$
\mathcal{L}=\{L_a: a\in \D\},  \quad L_a(f)=f\circ\vf_a-f(a),
$$
 with $\vf_a$ the M\"{o}bius automorphisms of $\D$. We can then apply the above theorem as
 in \cite[Example 4]{Perf2} to obtain information about the Volterra
 operators $T_g$ acting on the spaces, thus generalizing results from \cite{SiskZha}, \cite{bcdmps}, \cite{Laitila}, concerning
 compactness and weak compactness of $T_g$ on the pair $(VMOA, BMOA)$.

\begin{theorem}\label{thm7}
Let $1<p < 2$ and  $g$ analytic on $\D$
such that the Volterra operator $T_g: BMOA_p\to BMOA_p$  is bounded. Then the following  are equivalent:

\begin{enumerate}[(i)]
\item $T_g: VMOA_p \to VMOA_p$ is weakly compact\\
\item $T_g(BMOA_p) \subseteq VMOA_p $\\
\item $T_g: BMOA_p \to BMOA_p$ is weakly compact\\
\item The function $g$ satsifies
$$
 \lim_{|I|\to 0}\frac{\left(\log\frac{2}{| I |}\right)^p}{| I|}
\int_{S(I)}| g'(z) |^p (1-| z|^2)^{p-1}\,dm(z)=0.
$$
\end{enumerate}
\end{theorem}

\begin{proof}
Starting with the duality $(VMOA_p)^{**}=BMOA_p$ it is easy to verify that
 $(T_g|_{VMOA_p})^{**}=T_g$. Hence, the
equivalence of $(i), (ii)$ and $(iii)$ is immediate  due to Gantmacher's Theorem \cite[Theorem 5.23]{AlBu}.
We will prove the equivalence between $(iii), (iv)$.

Suppose $(iii)$ holds. Consider a sequence of arcs $\{I_n\}$ such that $|I_n|\to 0$, and let
 $w_n=(1-|I_n|)\zeta_n $ where $\zeta_n \in \partial \D$ is the center of the arc $I_n$.
Without loss of generality we may assume that $\lim\limits_{n\to \infty}w_n=\zeta\in \partial \D$.
Consider the functions
 $$
 f_n(z)=\log\frac{1}{1-\bar{w}_n z},\quad f_0(z)=\log\frac{1}{1-\bar{\zeta}z},
 \quad h_n(z)=\log\frac{1-\bar{\zeta}z}{1-\bar{w}_n z}\,,
 $$
 and notice that $\n{h_n}_{BMOA_p}\leq C\n{\log\frac{1}{1-z}}_{BMOA_p}$, and also that
 \begin{equation}\label{test-funct}
 \lim_{| I_n | \to 0}\n{h_n}_{D^p_{p-1}}=0.
 \end{equation}
  For $z \in S(I_n)$ we have that $| f_n(z)|\simeq\log\frac{2}{| I_n |}$, hence
\begin{align*}
\frac{\left(\log\frac{2}{|I_n|}\right)^p}{|I_n|} &\int_{S(I_n)} | g'(z)|^p (1-| z|^2)^{p-1}\,dm(z)\\
&\leq C \frac{1}{| I_n|} \int_{S(I_n)}| f_n(z) |^p | g'(z)|^p (1-| z |^2)^{p-1}\,dm(z)\\
&\leq C \frac{1}{| I_n|} \int_{S(I_n)}| f_0(z) |^p | g'(z)|^p (1-| z |^2)^{p-1}\,dm(z)\\
& \quad +C \frac{1}{| I_n|} \int_{S(I_n)}| f_n(z)-f_0(z) |^p | g'(z)|^p (1-| z |^2)^{p-1}\,dm(z)\\
&\leq C\frac{1}{|I_n |} \int_{S(I_n)}| (T_g(f_0)(z))' |^p  (1-| z |^2)^{p-1}\,dm(z)
+C\n{T_g(h_n)}^p_{BMOA_p}\\
&=A_n+B_n
\end{align*}

Since $f_0 \in BMOA_p$ from the equivalence of $(ii)$ and $(iii)$ we have that $T_g(f_0)\in VMOA_p$
thus  $\lim\limits_{n\to\infty}A_n= 0$. In addition by the hypothesis and the theorem mentioned  above
 \cite[Theorem 3.2]{Perf2},  for  $\e>0$ there is a constant $C$ such that
\begin{align*}
\n{T_g(h_n)}_{BMOA_p}&\leq C\n{h_n}_{{D^p_{p-1}}}+\e\n{h_n}_{BMOA_p}\\
&\leq C\n{h_n}_{{D^p_{p-1}}}+\e C'\n{\log\frac{1}{1-z}}_{BMOA_p}.
\end{align*}
In combination with (\ref{test-funct}) it follows then that $\lim\limits_{n\to\infty}B_n= 0$  and the desired condition $(iv)$ follows.

Conversely if  $(iv)$ holds then from \cite[Theorem 5.2(ii)] {PaZh}  $T_g$ is compact on $BMOA_p$ and $(iii)$  follows.
\end{proof}

We collect in a corollary the information  about Volterra operators on $BMOA_p$.
\begin{corollary}\label{collect}
Let $1 < p < 2$ and let $g$ be analytic on $\D$, such that the operator 
$T_g: BMOA_p\to BMOA_p$ is bounded.
Then the  following  are equivalent:
\begin{enumerate}[(i)] 
{\item $T_g: BMOA \to BMOA_p$ is compact\\

\item$T_g: BMOA_p \to BMOA_p$ is compact\\

\item $T_g: BMOA_p \to BMOA_p$ is weakly compact\\

\item $T_g: VMOA_p \to VMOA_p$ is weakly compact\\
 
\item $T_g(BMOA_p) \subseteq VMOA_p $\\

\item $T_g(BMOA) \subseteq VMOA_p $\\

\item The function $g$ satisfies
\[
\lim_{|I|\to 0}\frac{\left(\log\frac{2}{|I|}\right)^p}{|I|}
 \int_{S(I)} | g'(z)|^p (1-| z|^2)^{p-1}\,dm(z)=0.
\]}
\end{enumerate}
\end{corollary}

\subsection{An application to composition semigroups}

 We will apply the above results on $T_g$, with a specific choice of the symbol $g$,
  to characterize the maximal subspace of strong continuity
 $[\phi_t, BMOA_p]$   when the inducing semigroup of functions $\{\phi_t\}$ has
 its Denjoy-Wolff point inside the disc,  thus extending results in \cite[Corollary 2]{bcdmps}.

 Given a semigroup $\{\phi_t \}$ with  Denjoy-Wolff point $b$ and infinitesimal generator $G$, define the
 \textit{associated g-symbol} $\g(z)$ as follows,
\begin{enumerate}
\item If $b\in\D$ let  $\g(z)=\int_b^z \frac{\z -b}{G(\z)}\,d\z$
\item If $b\in \partial \D$ let $\g(z)=\int_0^z \frac{1}{G(\z)}\,d\z$,
\end{enumerate}
and  consider the Volterra operators $T_{\g}$.
Notice that if  $b\in \partial\D$, then  $\g$ coincides with the Koenigs function $h$ associated
to $ \{\phi_t \}$ \cite[page 234]{Sisk}. The function $h$ has the geometric property that for every $w\in h(\D)$, the half-line $\{w+t: t\geq 0\}$ is contained in $h(\D)$. Thus the range of $h$ contains an infinite strip, and an immediate consequence is that $h$ does not belong to the  little Bloch space $\mathcal{B}_0$.
 This implies that $T_{\g}$ cannot be  bounded on $BMOA_p$. For if we assume that $T_{\g}$ is
 bounded then  condition (\ref{logCar}) holds  for $g=\g$.
In particular, since $(\log\frac{2}{|I|})^p\to \infty $ as $|I|\to 0$ we must have
$$
\lim_{|I|\to 0}\frac{1}{| I |} \int_{S(I)} | \g'(z)|^p (1-|z|^2)^{p-1}\,dm(z)=0,
$$
thus  $h=\g\in VMOA_p\subset \mathcal{B}_0$ and this is a contradiction.

\begin{corollary}
Let $1<p < 2$ and $\{\phi_t\}$ a semigroup with Denjoy-Wolff point $b\in \D$ and associated $g$-symbol
$\g(z)$, such that $T_{\g}$ is bounded on $BMOA_p$.
Then $[\phi_t, BMOA_p]=VMOA_p$ if and only if $T_{\g}$ is (weakly) compact on $BMOA_p$.
\end{corollary}

\begin{proof}
Suppose  $T_\gamma$ is (weakly) compact. From  Corollary (\ref{collect}), condition
(vii)  holds with $\frac{z-b}{G}=\g'$, and consequently Theorem \ref{theNec}  gives the desired conclusion.

Conversely, suppose that $[\phi_t, BMOA_p]=VMOA_p$. Note that $f'/\g'=Gf'/(z-b)$ and that, since
$G(b)=0$, we have $Gf'\in BMOA_p$ if and only if $Gf'/(z-b)\in BMOA_p$.
By the assumption and from (\ref{ssc}), we have that
$$
VMOA_p = \overline{\{f\in BMOA_p:\;Gf'\in BMOA_p\}}
$$
Now if $f\in BMOA_p$ is such that $Gf'\in BMOA_p$, this is equivalent to $f'/\g'=m\in BMOA_p$
and further to
$$
f(z)=\int_0^zm(\z)\g'(\z)\,d\z +c, \quad c\, \,\text{a constant},
$$
for some  $m(z)\in BMOA_p$. Thus $f\in BMOA_p\cap (T_{\g}(BMOA_p)\oplus\C)$. But $T_{\g}$
is bounded on $BMOA_p$ so this intersection is just the image $T_{\g}(BMOA_p)$ plus the constants. Since $VMOA_p$ is closed, it follows that
$T_{\g}(BMOA_p)\subset VMOA_p$ and Corollary \ref{collect} then says that
 $T_{\g}$ is (weakly) compact on $BMOA_p$.
\end{proof}
\subsection{Acknowledgements} The authors would like to thank professor A.G. Siskakis for carefully reading this manuscript and for his usefull suggestions that helped to improve this article.

\bibliographystyle{amsplain}

\begin{thebibliography}{99}

\bibitem{Al.Sisk 1} A. Aleman and A. Siskakis, An integral operator on $H^p$,
Complex Variables, 28 (1995), 140-158.

\bibitem{AlBu} C. Aliprantis, O. Burkinshaw, \emph{Positive operators}, Springer, Dordrecht, (2006)

\bibitem{AJS} A. Anderson, M. Jovovic, W. Smith, \emph{Composition semigroups on BMOA and $H^\infty$},
J. Math. Anal. Appl, 449(1) (2017), 843-852

\bibitem{AuLa}R. Aulaskari, P. Lappan, \emph{Criteria for an analytic function to be Bloch and
a harmonic or meromorphic function to be normal}, Complex analysis and its Applications,
Pittman Research notes in Mathematics, Series 305, Longman, Harlow (1994), 136-146

\bibitem{AuTo}R. Aulaskari, L. Tovar, \emph{On the function spaces $B_p$ and $Q_p$},
Bull. Hong Kong Math. Soc., 1(1997), 203-208

\bibitem{AuXiZh}R. Aulaskari, J. Xiao, R. Zhao, \emph{On subspaces and subsets of BMOA and UBC},
Analysis 15(1995), 101-121


\bibitem{Berkson} E. Berkson, H. Porta, \emph{Semigroups of analytic functions and composition operators},
Michigan Math. J., 25(1978), 101-115

\bibitem{bcdms} O. Blasco, M. Contreras, S. D\'iaz-Madrigal, J. Martinez, A. Siskakis,
\emph{Semigroups of composition operators in BMOA and the extension of a theorem of Saranson},
Integr. equ. oper. theory, 61(2008), 45-62

\bibitem{bcdmps} O. Blasco, M. Contreras, S. D\'iaz-Madrigal, J. Martinez, M. Papadimitrakis,
A. Siskakis, \emph{Semigroups of composition operators and integral operators in spaces of analytic functions},
Ann. Acad. Sci. Fenn. Math., 38(2013), 1-23

\bibitem{bmoa}P. Bourdon, J. Cima, A. Matheson, \emph{Compact composition operators on BMOA},
Trans. Amer. Math. Soc. 351(1999), 2183-2196

\bibitem{cdmv}M. D. Contreras, S. D\'iaz-Madrigal, D. Vukoti\'c,
\emph{Compact and weakly compact composition operators
    from     the Bloch space into M\"{o}bius invariant spaces}, J. Math. Anal. Appl.,
    Volume 415, (2)(2014), 713-735


\bibitem{DM} N. Danikas,  C. Mouratides, \emph{Blaschke products in $Q_p$ spaces},
 Complex Variables Theory Appl. 43 (2000),  199--209.


\bibitem{Girela}D. Girela, \emph{Analytic functions of bounded mean oscilation}, Complex function spaces, 61-170,
Univ. Joensuu Dept. Math. Rep. Ser., 4(2001)

\bibitem{GiPe} D. Girela and J. A.  Pelaez, \emph{Carleson measures, multipliersand integration operators for spaces of Dirichlet type} J. Funct. Anal. 241 (2006), 334--358.

\bibitem{GPV} D. Girela, J.A. Pel\'aez and D.  Vukotic, \textit{Integrability of the derivative
of a Blaschke product}, Proc. Edinb. Math. Soc. (2) 50 (2007), 673--687.

  \bibitem{Laitila} J. Laitila, S. Miihkinen and P. Nieminen, \textit{Essential norms and weak
  compactness of integration operators} Arch. Math. (Basel) 97 (2011), 39--48.


\bibitem{LP} M. Lindstr\"{o}m and N. Palmberg, \emph{Duality of a large family of analytic function spaces},
 Ann. Acad. Sci. Fenn. Math. 32 (2007),  251--267.



\bibitem{Lueck}D. Luecking, \emph{Representation and Duality in Weighted Spaces of Analytic Functions},
Indiana University Mathematics Journal, 34(2) (1985), 319-336


\bibitem{PaZh}J. Pau, R.Zhao \emph{Carleson Measures, Riemann Stieltjes and Multiplication
operators  on a general family of function spaces}, Integral Equations and Operator Theory, 78(2014), 483-514

\bibitem{PerRat}F. Perez-Gonzalez, J. R\"{a}tty\"{a}, \emph{Inner functions in the M\"{o}bius
invariant Besov-type spaces}, Proc. Edin. Math. Soc., 52(2009), 751-770

\bibitem{PerRat2}F. Perez-Gonzalez, J. R\"{a}tty\"{a}, \emph{Univalent functions in Hardy,
Bergman and related spaces}, Journal d' Analyse Mathematique, 105(1)(2008), 125-148

\bibitem{Perf1} K. M. Perfekt, \emph{Duality and Distance Formulas in Spaces Defined by Means of Oscillation},
Arkiv f\"{o}r Matematik 51 (2013), 345-361.

\bibitem{Perf2} K. M. Perfekt, \emph{Weak Compactness of Operators Acting on o-O Type Spaces},
Bulletin of the London Mathematical Society 47 (2015), 677-685

\bibitem{Pomm} Ch. Pommerenke, \emph{Boundary Behaviour of Conformal Maps},
 Springer-Verlag, Berlin, 1992

\bibitem{Pom} C. Pommerenke, Schlichte funktionen und analytische funktionen von beschrankter
mittlerer oszillation, Comment. Math. Helv., 52 (1977), 591-602

\bibitem{Ratt}J. R\"{a}tty\"{a}, \emph{On some complex function spaces and classes},
Annales Acad. Sci. Fenn. Math. Diss., 124(2001), 1-73

\bibitem{Sar} D. Sarason, \emph{Functions of vanishing mean oscilation},
Trans. Amer. Math. Soc., 207(1975), 391-405

\bibitem{Sisk} A. Siskakis, \emph{Semigroups of composition operators on spaces of analytic
functions, a review}, Contemp. Math. 213(1998), 229-252


\bibitem{Sisk3} A. Siskakis, \emph{Semigroups of composition operators in Bergman spaces},
Bull. Austral. Math. Soc., 35(1987), 397-406

\bibitem{Sisk4} A. Siskakis, \emph{Semigroups of composition operators on the Dirichlet space},
Results Math., 30(1996), 165-173.

\bibitem{SiskZha} A.Siskakis, R. Zhao \emph{Voltera type operators on spaces of analytic functions},
 Contemp. Math., 232(1999), 299-312

\bibitem{Tjani}  M. Tjani, \emph{Compact composition operators on Besov spaces},
Trans. Amer. Math. Soc. 355 (2003),  4683--4698.

\bibitem{Vin} S. Vinogradov, \emph{Multiplication and division in the space of analytic functions
with area integrable derivative, and in some related spaces}, Journal
 of Mathematical Sciences, 87 (5) (1997), 3806-3827

\bibitem{WirthXiao} K. Wirths, J. Xiao, \emph{Recognizing $Q_{p,0}$ functions per
Dirichlet space structure}, Bull. Belg. Math. Soc. 8(2001), 47-59


\bibitem{Wu} Z. Wu, \emph{Carleson measures and multipliers of Dirichlet spaces},
J. Funct. Anal., 169(1999), 148-163


\bibitem{Xiao1} J. Xiao, \emph{Holomorphic Q classes}, Springer - Verlag,
Lecture notes in Mathematics, (2001)

\bibitem{Xiao2} J. Xiao, \emph{Geometric $Q_p$ functions}, Birkhauser Verlag, (2006)

\bibitem{YuTo} C. Yuan, C. Tong, \emph{On analytic Campanato and F(p,q,s) spaces},
Complex Anal. Oper. Theory, 12 (2018), 1845-1875

\bibitem{Zhao} R. Zhao, \emph{On a general family of function spaces} ,
Ann. Acad. Sci. Fenn. Math. Diss., 105(1996), 1-56

\end{thebibliography}

\end{document}